\documentclass[12pt,leqno]{article}
\usepackage{latexsym}
\usepackage{amsfonts,amsmath,amssymb}
\usepackage{enumitem}
\usepackage{amsthm}
\usepackage{graphics, epsfig}
\usepackage{xcolor}
\usepackage{marvosym}
\usepackage{mathtools}
\usepackage{esint}

\newtheorem{lemma}{Lemma}[section]
\newtheorem{corollary}{Corollary}[section]

\newtheorem{thm}{Theorem}[section]
\newtheorem*{prth1.1}{Proof of Theorem 1.1}
\newtheorem*{prth1.2}{Proof of Theorem 1.2}
\newtheorem*{prth1.3}{Proof of Theorem 1.3}
\newtheorem*{prcor1.1}{Proof of Corollary 1.1}
\numberwithin{equation}{section}
\numberwithin{theorem}{section}
\newtheorem{lem}[thm]{Lemma}
\numberwithin{proposition}{section}
\numberwithin{lemma}{section}
\numberwithin{corollary}{section}
\numberwithin{remark}{section}

\newcommand{\Tmax}{T_{\rm max}}

\newcommand{\ep}{\varepsilon}

\newcommand{\RN}{\mathbb{R}^N}
\newcommand{\R}{\mathbb{R}}

\def\Xint#1{\mathchoice
    {\XXint\displaystyle\textstyle{#1}}%
    {\XXint\textstyle\scriptstyle{#1}}%
    {\XXint\scriptstyle\scriptscriptstyle{#1}}%
    {\XXint\scriptscriptstyle\scriptscriptstyle{#1}}%
    \!\int}
\def\XXint#1#2#3{\setbox0=\hbox{$#1{#2#3}{\int}$}
    \vcenter{\hbox{$#2#3$}}\kern-0.5\wd0}

\def\dashint{\Xint{\raise4pt\hbox to7pt{\hrulefill}}}
\begin{document}
\thispagestyle{empty}
\setcounter{page}{1}
\noindent
\begin{center}
{\bf { \Large Behavior in time of solutions of a Keller--Segel system with flux limitation and source term
}}

\begin{center}
    \vspace*{1cm} 
{\bf 
{ M.Marras \footnote{ Dipartimento di Matematica e Informatica, Universit\'a di Cagliari, via Ospedale 72, 09124 Cagliari (Italy), mmarras@unica.it},
S.Vernier-Piro \footnote{ Facolt\'a di Ingegneria e Architettura, Universit\'a di Cagliari, Viale Merello 92, 09123 Cagliari (Italy), svernier@unica.it},
 T.Yokota \footnote{Department of Mathematics, Tokyo University of Science, 1-3, Kagurazaka, Shinjuku-ku, Tokyo 162-8601 (Japan), yokota@rs.tus.ac.jp} }}\\
\end{center}

\end{center}


\begin{abstract}
In this paper we consider radially symmetric solutions of the following parabolic--elliptic cross-diffusion system
\begin{equation*}
\begin{cases}
u_t = \Delta u  -  \nabla \cdot (u  f(|\nabla v|^2 )\nabla v) + g(u), &  \\[2mm]
0= \Delta v -m(t)+  u ,  \quad \int_{\Omega}v \,dx=0, &  \\[2mm]
u(x,0)= u_0(x), &
\end{cases}
\end{equation*}
in $\Omega \times (0,\infty)$, with $\Omega$ a ball in $\RN$, $N\geq 3$, under homogeneous Neumann boundary conditions, where $g(u)= \lambda u - \mu u^k$ , $\lambda >0, \ \mu >0$, and $ k >1$, $f(|\nabla v|^2 )= k_f(1+ |\nabla v|^2)^{-\alpha}$, $\alpha>0$, which describes gradient-dependent limitation of cross diffusion fluxes. The function $m(t)$ is the time dependent spatial mean of $u(x,t)$  i.e.  $m(t) := \frac 1 {|\Omega|} \int_{\Omega} u(x,t)  \,dx$. Under smallness conditions on $\alpha$ and $k$,
we prove that the solution $u(x,t)$ blows up  in $L^{\infty}$-norm at finite time $T_{max}$ and for some $p>1$ it blows up also in $L^p$-norm.  In addition a lower bound of blow-up time is derived. Finally, 
under largeness conditions on $\alpha$ or $k$,
we prove that the solution is global and bounded in time.

\end{abstract}


\section{Introduction} \label{Intro}

Let us consider the chemotaxis system with flux limitation with source term,
\begin{equation} \label{sys}
\begin{cases}
u_t = \Delta u  -  \chi \nabla \cdot (u  f(|\nabla v|^2 )\nabla v) +g(u),   \quad & x\in \Omega, \ t>0,\\[2mm]
0= \Delta v -m(t) +  u ,  \qquad \qquad \qquad &x \in \Omega, \ t>0, \\[2mm]
\frac{\partial u}{\partial \nu}=\frac{\partial v}{\partial \nu}= 0, \quad \qquad \qquad \qquad &x \in \partial \Omega, \ t>0, \\[2mm]
u(x,0)= u_0(x), \qquad  \qquad &x \in \Omega,
 \end{cases}
\end{equation}
with $\Omega$ a ball in $\mathbb{R}^N$, $N\geq 3$, $m(t) =  \frac 1{ |\Omega|}\int u(x,t) \,dx >0$, $\int_{\Omega} v \,dx=0$, 

\begin{equation}\label{f}
f(|\nabla v|^2 )= k_f(1+ |\nabla v|^2)^{-\alpha}
\end{equation}
with some $k_f >0$ and $\alpha >0$, 
\begin{equation}
\label{g-logistic}
g(u)= \lambda u -\mu u^k 
\end{equation}
with $\lambda >0, \ \mu >0$, and $ k >1$, $u_0$ is a given nonnegative function. \\

The chemotaxis model \eqref{sys} with $g(u)=0$ and $f( |\nabla v|^2)=1$ is just the classical 
Keller--Segel system (see \cite{KS_1970}), which permits the concentration phenomena to result in the possible blowing up of solutions, and has been extensively studied since 1970s, such as the existence of global bounded solutions and the detection of some solutions blowing up in either finite or infinite time, in a great number of literature (see \cite{BBTW}, \cite{ChMT},  \cite{CMTY}, \cite{IY}, \cite{MTV}, \cite{MVP}, \cite{MOT}, \cite{Nagai}, \cite{NT} and the references therein).  \\

We refer that in the case $f( |\nabla v|^2)=1$, $\chi >0$  with $g(u)=\lambda u - \mu u^k$ , $\lambda \geq 0, \ \mu \geq 0$, and $1< k< \frac 32+ \frac 1{2n-2}$, $\Omega$ a ball in $\mathbb {R}^N,$ with $N\geq 5$, Winkler in \cite{W-2011} proved that there exist initial data such that the radially symmetric solution blows up in finite time. In \cite{W-2018}, with  $\Omega$ a ball in $\mathbb{R}^N, N\geq 3, \lambda \in \mathbb {R}, \mu>0, k>1,$
and with $m(t)$ replaced by the function $v(x,t)$ in the second equation, under the assumption 
\begin{equation*} \label{k_w2018}
k < \begin{cases}
 \frac 7 6, \quad \quad \qquad  \  \ {\rm if} \ \  N\in \{3,4\}, \\[6pt]
1 + \frac 1 {2(N-1)},  \ \ {\rm if} \ \ N\geq 5, 
 \end{cases}
\end{equation*}
the author derived a condition on the initial data sufficient to ensure the occurrence of blowing up solutions in finite time.\\
The range of $k$ has been improved by Fuest in \cite{Fuest}, where a nonnegative initial datum $u_0$ has been constructed  such that the solution blows up in finite time when $\chi =1$, 
  
 \begin{equation*} \label{k_Fuest}
 \begin{cases}
1<k< \min\left\{ 2, \frac N 2\right\}, \  \ \mu >0, &{\rm for} \ \  N\geq 3, \\[6pt]
k=2, \quad  \qquad  \qquad  \quad  \mu\in \bigl(0,\frac{N-4}N\bigr),  \quad &{\rm for} \  \  N \geq 5. 
 \end{cases}
\end{equation*}
 The value $k=2$ is critical in the four and higher dimensions. \\
Recently  the case $f$ depending on the gradient of $v$ (flux limitation term) received considerable attention in the biomathematical literature.\\
The most relevant results on flux limitation  concern the case  $g(u)=0$.\\

In particular\\
$\diamond$ \  If   $f( |\nabla v|^2)=  |\nabla v|^{p-2} $, $\chi >0$, $\Omega\subset  \mathbb{R}^N, $
$$p\in(1,\infty) \quad  { \rm for}   \ N=1;  \quad \quad   
p\in \Bigl(1, \ \frac N{N-1}\Bigr) \quad  { \rm for} \   N\geq 2,$$ 
Negreanu and Tello in \cite{NT} obtained uniform bounds in $L^{\infty} (\Omega)$ and  the existence of global in time solutions; for the one-dimensional case there exist infinitely many non-constant steady-states for $p\in (1,2)$.
\\
$\diamond$ 
If  $f( |\nabla v|^2)=  \frac{1 } { \sqrt{1+ |\nabla v|^2  } } $ and $\Delta u$ is replaced by $\nabla \cdot \bigl( \frac { u \nabla u  }{\sqrt{u^2+ |\nabla u|^2}}\bigr)$, Bellomo and Winkler \cite{BW2} obtained the global existence of bounded classical solutions for arbitrary positive radial initial data $u_0 \in C^3(\overline{\Omega})$ when
\begin{align*}
\int_{\Omega} u_0 < \frac1{\sqrt{(\chi^2 -1)_+ } } , \ \ {\rm if} \ N=1; \qquad   \chi <1, \ \ N\geq 2.
\end{align*}
In \cite{BW}, the authors shows that the above conditions are essentially optimal in the sense that if $\chi >1$ and

\begin{equation*}
m>\frac1{\sqrt{\chi^2 -1 } } , \ \ {\rm if} \ N=1; \qquad m>0 \ {\rm arbitrary,} \ \ { \rm if} \  N\geq 2 \\[6pt]
\end{equation*}

there exists $u_0\in C^3(\overline{\Omega})$ with $\int_{\Omega} u_0=m,$ such that there exists a a unique blowing up classical solution.\\

$\diamond $ \ If
 $f(|\nabla v|^2)\geq K_f \bigl( 1+ |\nabla v|^2   \bigr)^{-\alpha}, \ K_f>0$, $\chi=1$,   $0< \alpha  < \frac { N-2} {2(N-1) }$, $\Omega$ a ball in $\mathbb {R}^N,$ with $N\geq 3$,
 for a considerably large set of radially symmetric initial data, the problem admits solutions blowing up in finite time in $L^\infty$-norm for the first component.  Otherwise, if $f(|\nabla v|^2)\leq K_f \bigl( 1+ |\nabla v|^2   \bigr)^{-\alpha}$, $\chi=1$ and $\alpha$ satisfies 
 
 \begin{equation*} \label{alpha Winkler}
 \begin{cases}
 \alpha > \frac { N-2} {2(N-1) }, \   &{\rm for} \  N \geq 2,  \\[6pt]
 \alpha \in \mathbb{R}, \ \qquad     &{\rm for}\  N=1,
 \end{cases}
\end{equation*}
 in general (not  symmetric setting), a global bounded solution exists (\cite{W-2020}).\\ The case $\alpha=  \frac { N-2} {2(N-1) }$ plays the role of a critical exponent and it is still an open problem.

$\diamond $ 
If $ f(|\nabla v|^2)=K_f \bigl( 1+ |\nabla v|^2   \bigr)^{-\alpha}, \ K_f>0$, $\chi=1$, 
$0< \alpha  < \frac { N-2} {2(N-1) }$, $\Omega =B_R(0)\subset \mathbb {R}^N,$ with $N\geq 3$, Marras, Vernier-Piro and Yokota in \cite{MVY}, for suitable  initial data, proved that a solution which blows up in $L^\infty$-norm blows up also in $L^p$-norm for some $p>\frac N2.$  
Moreover, a safe time interval of existence of the solution $[0,T]$ is obtained, with $T$ a lower bound of the blow-up time.
\\
Less attention was payed to the case with $f$ depending on the gradient of $v$  in presence of a source term $g(u)$.\\
It is the purpose of the present paper to address the above question for a class of functions $g(u)$ modeling sources of logistic type: $g(u)=\lambda u - \mu u^k$ , $\lambda >0, \ \mu >0$, and $ k >1$.\\

{\bf Main Results} The present work is addressed to study the behavior in time of the solutions of problem \eqref{sys} with $\chi=1$ in presence of the flux limitation term and the source term $g(u)=\lambda u - \mu u^k$ to varying $k\in(1,2]$. In particular in Section \ref{BlowUp in L^{infty}}  we construct an initial data such that the solution of problem \eqref{sys} blows up in $L^{\infty}$-norm in the following sense.  \\

\begin{thm} [Finite-time blow-up in $L^\infty$-norm]\label{BULinfty}
Let $\Omega\equiv B_R(0) \subset \R^N$, $R>0$. 
Moreover suppose
\begin{align*}
&N\geq 3, \quad & k\in \Bigl(1,\, \min \Big\{2, 1+ \frac{(N-2)^2}{4} \Big\} \Bigr) \quad \ \   and \ \mu >0\\
 or \ &N\geq 5, &k=2  \quad \ \  and \ 0<\mu \leq \mu_0,
\end{align*}

where $\mu_0>0$ is a constant determined in Lemma \ref{lemma H_7}.
Assume 
\begin{equation}\label{alpha}
0<\alpha < \frac {N-2}{2(N-1)}.
\end{equation}

Then for all $m_0>0$ there exist  radially symmetric as well as radially decreasing initial data
\begin{align}\label{u_0}
u_0 \in C^0(\bar \Omega), \quad u_0 \not\equiv 0 
\end{align}
 such that  
 \begin{align*}
\frac 1 {|\Omega|}  \int_{\Omega} u_0 \,dx = m_0,
 \end{align*}

and such that \eqref{sys} possesses a unique classical solution $(u,v)$ in $\Omega \times (0,T_{max})$, for some $T_{max} \in(0, \infty)$, which blows up at $T_{max}$ in the sense that
\begin{align}\label{blowupinfty}
\limsup_{t \nearrow T_{max}} \|u(\cdot, t)\|_{L^\infty(\Omega)}=\infty.
\end{align}
\end{thm}

%

The second purpose of this paper is to prove that the solutions of \eqref{sys} blow up at finite time in $L^p$-norm, for some $p>1$, if they blow up in $L^{\infty}$-norm (Section \ref{BlowUp in L^p}).\\  

\begin{thm}[Finite-time blow-up in $L^p$-norm]\label{BULp}
Let\/ $\Omega \equiv B_R(0) \subset\mathbb{R}^N$, $N\geq 3$ and $R>0$. 
Then, a classical solution $(u, v)$ of \eqref{sys}  for 
$t \in (0, T_{max})$,
provided by Theorem~\ref{BULinfty}, is such that for all $p>\frac{N}{2}$, 
\begin{align*}
\limsup_{t \nearrow T_{max}} 
\left\|u(\cdot,t)\right\|_{L^{p}(\Omega)} 
= \infty.
\end{align*}
\end{thm}

The investigation on blow-up solutions of system \eqref{sys} goes on with the study of the behavior near the blow-up time $T_{max}$ (Section \ref{lower bound}). The goal is to obtain a safe time interval $(0,T)$, ($T<T_{max}$), of existence of the solutions of \eqref{sys}; to this end, we define, for all $p>1$, the auxiliary function
\begin{equation}\label{Psi} 
\Psi(t):= \frac 1 {p} \|u(\cdot,t)\|^{p}_{L^{p}(\Omega)} \quad {\rm with}\quad  \Psi_0 := \Psi(0)= \frac 1 {p} \|u_0\|^{p}_{L^{p}(\Omega)},
\end{equation}
and we determine a lower estimate of the blow-up time $T_{max}$.

\begin{thm}[Lower bound of blow-up time]\label{LB}
Let\ $\Omega \equiv B_R(0) \subset\mathbb{R}^N$, $N\geq 3$, $R>0$ and let $\Psi$ be defined in \eqref{Psi}. 
Then, for all  $p>\frac N 2$ and some positive constants $B_1, B_2, B_3, B_4$, 
the blow-up time $T_{max}$ for \eqref{sys}, provided by Theorem~\ref{BULinfty}, 
satisfies the estimate
\begin{align}\label{lower Tmax in Lp}
T_{max}\geq T:= \int_{\Psi_0}^{\infty}\frac{d\eta}{B_1 \eta + B_2\eta^{\gamma_1} + B_3 \eta^{\gamma_2} +B_4 \eta^{\gamma}},
  \end{align}
with $\gamma_1:= \frac{p+1}{p}, \ \  \gamma_2:=\frac{2(p+1)-N} {2p-N}, \ \ \gamma:=\frac{2(p+1) - \frac{N(p+1)(1+\epsilon)}{p+1+\epsilon} }{2p-\frac{N(1+\epsilon)(p+1)}{p+1+\epsilon}}, \ \  
0<\epsilon<\frac{2p}{N}-1$.
\end{thm}
\begin{corollary}\label{CorollaryLower bound}
Under the assumptions of Theorem \ref{BULp},
let $(u,v)$ be a solution of \eqref{sys} and $\Psi(t)$ and $\Psi_0$ defined in \eqref{Psi}.  Then there exists a safe interval of existence of  $(u,v)$ say  $[0,T]$ with 
\[
T := \frac{1}{\mathcal{A} (\gamma-1) \Psi_0^{\gamma -1}} \leq T_{max}.
\]
We remark that $ \frac{1}{ \mathcal{A} (\gamma -1) \Psi_0^{\gamma -1}} $ is explicitly computable.
\end{corollary}

We observe that  the blow-up phenomena can be avoided for different choises of the data.
Moreover, we will prove that the results in Theorem \ref{BULinfty}  with 
$f(|\nabla v|^2 )= k_f(1+ |\nabla v|^2)^{-\alpha}$ fulfilling $0<\alpha<\frac {N-2}{2(N-1)}$ and $\kappa \le 2$
cannot be improved. In fact if $\alpha>\frac {N-2}{2(N-1)}$ or $\kappa>2$ we obtain that  the global solution is bounded (Section \ref{boundedness}).\\

\begin{thm}[Global existence and boundedness]\label{GEB}
Let\ $\Omega \equiv B_R(0) \subset\mathbb{R}^N$, $N\geq 3$, $R>0$. 
Assume that either one of the following two conditions is satisfied: 
\begin{enumerate}
\item $\alpha>\dfrac {N-2}{2(N-1)}$ and $k>1$, 
\item $\alpha>0$ and $k>2$.
\end{enumerate}
Then for all radially symmetric nonnegative initial data $u_0 \in C^0(\bar{\Omega})$, 
system \eqref{sys} possesses a unique global classical solution $(u,v)$ 
in $\Omega \times (0,\infty)$, which is bounded in the sense that
\begin{align*}
\sup_{t \in (0,\infty)} \|u(\cdot, t)\|_{L^\infty(\Omega)}<\infty.
\end{align*}
\end{thm}


\section{Preliminaries}\label{preliminaries}
In this section, we present some preliminary lemmata which we shall use in the proof of our main results.

\begin{lemma}\label{lemma BU}
Let $N\geq 1$, and assume that $\Omega =B_R(0) \subset \mathbb{R}^N$  for some
$R>0$, $f$, $g$ satisfy \eqref{f}, \eqref{g-logistic} and that $u_0 \in C^0(\bar \Omega)$ is nonnegative and radially symmetric with respect to $x=0$. Then there exist $T_{max} \in (0, \infty]$ and a unique pair 
\[
(u,v) \in \Big((C^0(\bar \Omega \times [0, T_{max})) \cap C^{2,1} (\bar \Omega \times (0, T_{max} ))\Big)^2
\]
which solves \eqref{sys}  in the classical sense in $\Omega \times (0, T_{max}).$ Moreover, we have $u>0$ in $\Omega \times (0,T_{max})$, and both $u(\cdot , t)$ and $v(\cdot , t)$ are radially symmetric with respect to $x=0$ for all $t\geq 0$. Finally,
\[
\text{if} \ \ T_{max} < \infty, \ \ \text{then} \ \ 
\limsup_{t \nearrow T_{max}}\|u(\cdot, t) \|_{L^{\infty}(\Omega)} = \infty.
\]
\end{lemma}

%
%

We next give some properties of the Neumann heat semigroup 
which will be used later. 
For the proof, see \cite[Lemma 2.1]{Cao} and \cite[Lemma~1.3]{W-2010}.

\begin{lemma} \label{Cao} 
Let $(e^{t \Delta})_{t\geq 0}$ be the Neumann heat semigroup in 
$\Omega$, and let $\mu_1 >0$ denote the first non zero eigenvalue of 
$-\Delta$ in $\Omega$ under Neumann boundary conditions. 
Then there exist $k_1,  k_2 >0$ which 
 depend only on $\Omega$ and have  
the following properties\/{\rm :}
\begin{enumerate}[label=(\roman*)]
\item if\/ $1 \leq q\leq {\rm p}\leq \infty$, then 
\begin{equation} \label{etDeltaz}
\|e^{t \Delta} z\|_{L^{{\rm p}}(\Omega)} \leq k_1\bigl(1+t^{ - \frac N 2(\frac 1 q - \frac 1 {\rm p})}\bigr) e^{-\mu_1 t} 
\|z\|_{L^q(\Omega)}, \ \ \forall \,t >0
\end{equation}
holds for all $z\in L^q(\Omega)$ satisfying $\int_{\Omega} z =0$.
\item If\/ $1< q \leq {\rm p} \leq \infty$, then
\begin{equation} \label{etDelta nablaz}
\|e^{t \Delta} \nabla \cdot \textbf{z\,}\|_{L^{{\rm p}}(\Omega)} \leq k_2\big(1+ t^{-\frac 1 2 - \frac N 2(\frac 1 q - \frac 1 {\rm p})}\big) e^{-\mu_1 t} \|\textbf{z\,}\|_{L^q(\Omega)}, \ \ \forall \,t >0
\end{equation}
is valid for any $\textbf{z\,} \in (L^{q}(\Omega))^N$, 
where $e^{t \Delta} \nabla \cdot {}$ is the extension of the operator
$e^{t \Delta} \nabla \cdot {}$ on $(C_0^\infty(\Omega))^N$ to $(L^q(\Omega))^N$. 
\end{enumerate}
\end{lemma}

We observe that since constants are invariant under $e^{t \Delta}$ we can use \eqref{etDeltaz} writing $\bar z = \frac 1 {|\Omega|}\int_{\Omega} z\,dx$ so that we have $\int_{\Omega} (z - \bar z)\,dx=0$ (see \cite{W-2010})  .\\

\begin{lemma}\label{lemma int u < bar m}
Let $\Omega \subset \mathbb{R}^N, \ N\geq 1$, be a bounded and smooth domain, and $\lambda>0$, 
$\mu >0$, $k>1$. Then for a solution $(u,v)$ of \eqref{sys} we have 
\begin{equation}\label{int u < bar m}
\int_{\Omega} u \,dx \leq \bar{m}, \ \ \text{for all} \ t \in(0,T_{max}),
\end{equation}
with 
\begin{equation}\label{bar m}
\bar{m}= \max \Big \{\int_{\Omega} u_0\,dx, \ \Bigl(\frac{\lambda }{\mu}\ |\Omega |^{k-1}\Bigr)^{\frac{1}{k-1}}\Big \}.
\end{equation}
\end{lemma}
\begin{proof}
From the first equation in \eqref{sys} we obtain
\begin{equation}\label{max princ bar m}
\begin{split}
\frac{d}{dt}\int_{\Omega} u\, dx = \lambda \int_{\Omega} u \,dx - \mu \int_{\Omega} u ^ k \,dx \leq \lambda \int_{\Omega} u \,dx - \mu |\Omega |^{1-k}\Big (\int_{\Omega} u  \,dx \Big)^{k}
\end{split}
\end{equation}
where, in the last term we used H$\ddot{\rm o}$lder's inequality: $\int_{\Omega} u \,dx\leq |\Omega|^{\frac{k-1}{k}} \Big(\int_{\Omega} u^k \,dx\Big)^{\frac 1 k}$.\\
From \eqref{max princ bar m} we infer that $z=\int_{\Omega} \; u dx$ satisfies 
\[
\begin{cases}
z' (t) \leq \lambda z(t) - \bar{\mu} z^k(t), \ \ \ \ \bar{\mu} =\mu |\Omega |^{1-k},  \ \ \ \text{for all} \ t\in [0,T_{max}), &\\[2pt]
z(0)=z_0. &
\end{cases}
\]
Upon an ODE comparison argument this entails that
\begin{equation*} \label{z < bar m}
z(t)\leq \bar{m}, \ \ \text{for  all} \ t \in(0,T_{max}).
\end{equation*}
This clearly proves the lemma.
\end{proof}

In Section \ref{lower bound} we will use the Gagliardo--Nirenberg inequality in the following form.
\begin{lemma}\label{lemma GN ineq}
Let\/ $\Omega$ be a  bounded and smooth domain of\/ $\mathbb{R}^N$ with $N\geq 1$. 
Let $ \mathsf{r}\geq 1$, $1\leq\mathsf{q}< \mathsf{p}\leq \infty$, $ \mathsf{s}>0$.
	Then there exists a constant
	$C_{{\rm GN}}>0$ such that 
\begin{equation} \label{GN ineq}	
		\|f\|^{\mathsf{p}}_{L^{ \mathsf{p}}(\Omega)}\leq C_{{\rm GN}} \Big(\|\nabla f\|^{\mathsf{p} a}_{L^{ \mathsf{r}}(\Omega)}
		\|f\|_{L^{ \mathsf{q}}(\Omega)}^{{\mathsf{p}}(1-a)}
	+\|f\|^{\mathsf{p}}_{L^{ \mathsf{s}}(\Omega)}\Big)
	\end{equation}
for all $f\in L^{\textsf{q}}({\Omega})$ with 
$\nabla f \in (L^{\textsf{r}}(\Omega))^N$
and
$a:=\frac{\frac{1}{ \mathsf{q}}-\frac{1}{ \mathsf{p}}}
		{\frac{1}{ \mathsf{q}}+\frac{1}{N}-
		\frac{1}{ \mathsf{r}}} \in (0,1)$. 	
\end{lemma}
\begin{proof}
Following from the Gagliardo--Nirenberg inequality 
(see \cite{Nir} for more details): 
\begin{equation*}\label{GNfirst} 	
		\|f\|^{ \mathsf{p}}_{L^{ \mathsf{p}}(\Omega)}\leq \Big[c_{{\rm GN}} \Big(\|\nabla f\|^{a}_{L^{ \mathsf{r}}(\Omega)}
		\|f\|_{L^{ \mathsf{q}}(\Omega)}^{1-a} +\|f\|_{L^{ \mathsf{s}}(\Omega)}\Big) \Big]^{ \mathsf{p}},
	\end{equation*}
with some $c_{{\rm GN}}>0$, and then from the inequality
\begin{equation*}
(\mathsf{a}+\mathsf{b})^{\mathsf{p}} \leq 2^{\mathsf{p}}(\mathsf{a}^{\mathsf{p}} + 
\mathsf{b}^{\mathsf{p}})\quad {\rm for\ any}\ \mathsf{a}, \mathsf{b}\geq 0,
\ \mathsf{p}>0,
\end{equation*}
we arrive to \eqref{GN ineq} with $C_{\rm GN}= 2^{\mathsf{p}} c_{\rm GN}^{\mathsf{p}}$.
\end{proof}

\begin{lemma}\label{lemma y(t)}
Let $\beta >0$, $\delta >0$, $\gamma >0$ and suppose that for some $T>0$,  $y\in C^0([0,T])$ is a nonnegative function satisfying 
\begin{align*}
y(t) \geq \beta + \delta \int_0^t y^{1 + \gamma}  (\tau) \,d \tau \quad  \forall\, t\in(0,T).
\end{align*}
Then $T\leq \frac 1 {\gamma \delta \beta^{\gamma}}.$
\end{lemma}
For the proof see \cite[Lemma 2.4]{W-2011}.


\section{Blow-up in $L^{\infty}$-norm}\label{BlowUp in L^{infty}}

{\bf Transformation in nonlocal scalar parabolic equation:}\\
Assume $\Omega =B_R(0)$, $R>0$ and  $u_0\in C^0(\bar \Omega)$ is radially symmetric with respect to $x=0$. If $(u,v)$ is the corresponding radial solution in $\Omega \times (0,T_{max})$ asserted by Lemma~\ref{lemma BU}, we write $u=u(r,t)$ and $v=v(r,t)$ with $r=|x|\in [0, R]$.\\
Following J$\rm {\ddot a}$ger--Luckhaus (\cite{JL}) we introduce the mass accumulation function
\begin{equation} \label{w}
w(s,t):= \int_0^{s^{\frac 1 N}} \rho^{N-1} u(\rho, t) \,d \rho, \ \ s= r^N \in [0,R^N], \ \ t\in[0, T_{max}).
\end{equation}
We have
\begin{equation*} \label{w_s}
w_s(s,t)= \frac 1 N u(s^{\frac 1 N} , t) \geq 0, \quad w_{ss} (s,t) = \frac 1 {N^2} s^{\frac 1 N - 1} u_r(s^{\frac 1 N} , t).
\end{equation*}
From the second equation in \eqref{sys} we deduce
\begin{equation*}\label{v_r}
\frac 1 {r^{N-1}} \big(r^{N-1} v_r(r,t) \big)_r= m(t) -  u
\end{equation*}
and
\begin{equation*}\label{v_r bis}
r^{N-1} v_r(r,t)=  m(t) \int_0^{r} \rho^{N-1} \,d\rho  -  \int_0^{r} \rho^{N-1} u(\rho, t) \,d\rho =  \frac {m(t) r^N}{N} - \int_0^{r} \rho^{N-1} u(\rho, t) \,d\rho.
\end{equation*}
Using \eqref{sys} we obtain 
\begin{equation*}\label{w_t}
\begin{aligned}
w_t(s,t)= \,& \int_0^{s^{\frac 1 N}} \rho^{N-1} u_t(\rho, t) \,d \rho \\
= &\int_0^{s^{\frac 1 N}} \big( \rho^{N-1} u_r \big)_r (\rho, t) \,d \rho - \int_0^{s^{\frac 1 N}} \Big( \rho^{N-1} u(\rho, t) v_r f( v^2_r) \Big)_r \,d \rho \\
&+ \lambda \int_0^{s^{\frac 1 N}} \rho^{N-1} u(\rho, t) \,d \rho - \mu\int_0^{s^{\frac 1 N}} \rho^{N-1} u^k(\rho, t) \,d \rho \\
= &\, s^{1 - \frac 1 N} u_r(s^{ \frac 1 N},t) - s^{1 - \frac 1 N} u v_r f(v^2_r(s^{\frac 1 N},t))\\ 
&+ \lambda \int_0^{s^{\frac 1 N}} \rho^{N-1} u(\rho, t) \,d \rho - \mu\int_0^{s^{\frac 1 N}} \rho^{N-1} u^k(\rho, t) \,d \rho\\
= &\,N^2 s^{2-\frac 2 N} w_{ss} + N w_s \Bigl(w- \frac {m(t)} N s\Bigr) f\Bigl(s^{\frac 2 N -2} (w- \frac {m(t)} N s)^2\Bigr)\\
&+ \lambda w - \mu N^{k-1} \int_0^{s}w^k_s (\sigma,t)\,d\sigma\\
\end{aligned}
\end{equation*}
%
%
and 
 \begin{equation} \label{parabproblem}
\begin{cases}
w_t = N^2s^{2-\frac 2 N} w_{ss} + N(w-\frac{m(t)} N s ) w_s f \big(s^{\frac 2 N -2} (w- \frac{m(t)}{N} s)^2\big) \\[6pt]
\qquad + \lambda w-\mu N^{k-1}  \int_0^s w^k_s (\sigma,t)\,d\sigma,
    \ \ \  s \in (0, R^N) , \ t \in (0, T_{max}),\\[6pt]
w(0,t)=0, \ \ \   \ \  w(R^N,t)= \frac{\mu R^N}{N},   \ \ \  t \in (0, T_{max}), \\[6pt]
w(s,0)=w_0(s), \ \ \  s\in (0,R^N)
\end{cases}
\end{equation}
with $w_0(s)= \int_0^{s^{\frac 1 N}} \rho^{N-1} u_0(\rho) d \rho , \ \ s\in[0,R^N]$.\\

Our aim is to prove that the functional $\int_0^{R^N} s^{-a}  w^ b (s,t) \,ds$, for suitable $a\in(0,1)$ and $b\in(0,1)$ blows up in finite time.\\
To this end, we use the estimate $w_s \leq \frac w s$ proved by Fuest (\cite[Lemma 3.3]{Fuest}):

\begin{lemma}\label{Fuest}
Assume that $u_0$ satisfies \eqref{u_0}. For all $s\in[0, R^N]$ and $t\in(0,T_{max})$,
\begin{equation}\label{w_s<w/s}
w_s(s,t) \leq \frac{w(s,t)} s \leq w_s(0,t)
\end{equation}
holds.
\end{lemma}
\begin{proof}
By a similar way as in \cite[Lemma 2.3]{BW2} where $\alpha = \frac 1 2$  
and as in \cite[Lemma 3.7]{Fuest2}, 
we can show that $u_r \leq 0$ in $(0,R)\times (0, T_{max})$ and following the steps in \cite{Fuest} we arrive to \eqref{w_s<w/s}.
\end{proof}
The next step is to prove that the functional $\int_0^{R^N} s^{-a}  w^ b (s,t) \,ds$ satisfies a differential inequality. First we obtain the following estimate.

\begin{lemma}\label{lemma s^-aw^b}
Assume Lemma \ref{lemma int u < bar m} and $\Omega = B_R(0)\subset \mathbb{R}^N$ with some $R>0$ and $N\geq 2$. Let $u_0\in C^0(\bar \Omega)$ be radial, and let $(u,v)$ denote the solution of \eqref{sys} in $\Omega \times(0, T_{max})$. Then for all $a>0$ and $b\in (0,1)$, the function $w$ defined in \eqref{w} satisfies
\begin{align}\label{s^-aw^b}
\notag \frac 1 b \int_0^{R^N} s^{-a} w^b(s,t) \,ds 
\,\geq\, &\frac 1 b \int_0^{R^N} s^{-a} w_0^b(s)\, ds  \\
\notag &- k_f \bar{m}|\Omega|^{-1} \int_0^t   \int_0^{R^N} s^{1-a} w^{b-1} w_s \,ds d\tau\\
\notag &+ \frac{a N k_f}{2(b+1)} \bar C \int_0^t \int_0^{R^N}s^{-a-1} w^{b+1} \,ds d\tau\\
\notag & +  \frac 1 2 N k_f  \bar C \int_0^t \int_0^{R^N} s^{-a} w^{b} w_s  \,ds d\tau \\
\notag &+ N^2  (1-b)\int_0^t \int_0^{R^N}  s^{2- \frac 2N  -a }w^{b-2} w^2_{s}\,ds d\tau\\
\notag & - 2N(N-1) \int_0^t  \int_0^{R^N} s^{1- \frac 2N -a} w^{b-1} w_s \,ds d\tau\\
&- \mu N^{k-1}\int_0^t \int_0^{R^N} s^{-a} w^{b-1} \Bigl(\int_0^s w_s^k d\sigma \Bigr) \,ds d\tau,
\end{align}
with $\bar C:=  \Big[\frac{N^{2}}{N^2 + 2|\Omega|^{-2}{\bar m}^2 R^2}\Big]^{\alpha}$, and ${\bar m}$ in \eqref{bar m}
\end{lemma}

\begin{proof}
Following the steps in \cite[Lemma 2.1]{W-2011} we  multiply the first equation in \eqref{parabproblem}  by $(s+\epsilon)^{-a} w^{b-1}(s,\tau)$ , $\epsilon >0$, and integrate over $s\in(0,R^N)$. We obtain 
\begin{align}\label{w+eps}
\notag &\frac 1 b \frac d {dt} \int_0^{R^N} (s+\epsilon)^{-a} w^b(s,t) \,ds \\
\notag &\geq N^2 \int_0^{R^N}  s^{2- \frac 2N }(s+\epsilon)^{-a} w^{b-1} w_{ss}\,ds \\
\notag&\quad\ + N\int_0^{R^N} (s+\epsilon)^{-a} w^{b-1} w_s\Bigl(w- \frac{m(t)}{N}s\Bigr) f\Bigl(s^{\frac 2 N -2}\Bigl(w-\frac{m(t)}N s\Bigr)^2\Bigr)\,ds\\
&\quad\ - \mu N^{k-1}\int_0^{R^N} (s+\epsilon)^{-a} w^{b-1} \Bigl(\int_0^s w_s^k \,d\sigma \Bigr) \,ds = \mathcal{I}_1 + \mathcal{I}_2 + \mathcal{I}_3.
\end{align}
Integrating by part we have
\begin{align}\label{I_1}
\notag \mathcal{I}_1&= N^2 \int_0^{R^N}  s^{2- \frac 2N }(s+\epsilon)^{-a} w^{b-1} w_{ss}\,ds \\
\notag &= N^2s^{2- \frac 2N }(s+\epsilon)^{-a} w^{b-1} w_{s}\big|_{0}^{R^N}  - N^2  (b-1)\int_0^{R^N}  s^{2- \frac 2N }(s+\epsilon)^{-a} w^{b-2} w^2_{s}\,ds\\ 
\notag & \quad - N^2 \int_0^{R^N}\frac d {ds} \big( s^{2- \frac 2N }(s+\epsilon)^{-a}\big) w^{b-1} w_{s} \,ds \\
\notag &\geq N^2  (1-b)\int_0^{R^N}  s^{2- \frac 2N }(s+\epsilon)^{-a} w^{b-2} w^2_{s}\,ds \\
 & \quad - 2N(N-1)  \int_0^{R^N} s^{1- \frac 2N} (s+ \epsilon)^{-a} w^{b-1} w_s \,ds
\end{align}
where in the last step we used  $\frac d{ds} \big( s^{2- \frac 2N }(s+\epsilon)^{-a}\big) =(2- \frac 2 N ) s^{1- \frac 2 N} (s+\epsilon)^{-a} - a s^{2- \frac 2 N}(s+\epsilon)^{-a-1} \leq (2-\frac 2 N) s^{1-\frac 2N}(s+\epsilon)^{-a}$.
\vskip.2cm
In $\mathcal{I}_2$ we have
\begin{align*}
\notag \mathcal{I}_2&= N\int_0^{R^N} (s+\epsilon)^{-a} w^{b-1} w_s\Bigl(w- \frac{m(t)}{N}s\Bigr) f\Bigl(s^{\frac 2 N -2}\Bigl(w-\frac{m(t)}N s\Bigr)^2\Bigr)\,ds\\
\notag &= N\int_0^{R^N} (s+\epsilon)^{-a} w^{b} w_s f\big(s^{\frac 2 N -2}(w-\frac{m(t)}N s)^2\big)\,ds \\
&\quad - \int_0^{R^N} s (s+\epsilon)^{-a} w^{b-1} w_s m(t) f\Bigl(s^{\frac 2 N -2}\Bigl(w-\frac{m(t)}N s\Bigr)^2\Bigr)\,ds=\mathcal{I}_{21} + \mathcal{I}_{22}.
\end{align*}
Taking into account that $u\geq 0$ we have $w_s \geq 0$ in $(0, R^N) \times (0,T_{max})$ and from the boundary condition at $s=R^N$ we have $w(s,t) \leq \frac{m(t) R^N}{N}$ for all $s\in[0, R^N]$ and $t\in [0, T_{max})$.
\vskip.2cm
By using $w\leq \frac{m(t) R^N}{N}$ and $s\leq R^N$, using \eqref{int u < bar m} we arrive at 
\[
 \Bigl(\frac {m(t)}N s - w\Bigr)^2 \leq \frac{m^2(t)}{N^2} s^2 + w^2 \leq  2 \frac{m^2(t)}{N^2} R^{2N} \leq 2 \frac{|\Omega|^2\bar m^2}{N^2} R^{2N}:= \bar{M}^2
\]
so that
\[
  f\Bigl(s^{\frac 2 N -2}\Bigl(w-\frac{m(t)}N s\Bigr)^2\Bigr)= k_f \frac 1 {\big[1+ s^{\frac 2 N -2} ( \frac {m(t)}N s - w) ^2\big]^{\alpha}}\geq k_f \frac 1 {\big[1+ \bar{M}^2\big]^{\alpha}}.
\]

We now split $\mathcal{I}_{21}= \frac{\mathcal{I}_{21}} 2 + \frac{\mathcal{I}_{21}}2$. 
Computing 
\begin{align*}
\notag \frac{\mathcal{I}_{21} }{2}
&= \frac 1 2 N k_f  \int_0^{R^N} (s+\epsilon)^{-a} w^{b} w_s f\Bigl(s^{\frac 2 N -2}\Bigl(w-\frac{m(t)}N s\Bigr)^2\Bigr)\,ds\\
& \geq   \frac 1 2 N k_f  \int_0^{R^N} (s+\epsilon)^{-a} w^{b} w_s \frac 1 {\big[1+ \bar{M}^2\big]^{\alpha}}\, ds
\end{align*}
and integrating by parts we get
\begin{align*}
\frac 1 2 N k_f  \int_0^{R^N} (s+\epsilon)^{-a} w^{b} w_s \frac 1 {\big[1+\bar{M}^2\big]^{\alpha}} \,ds
\notag & = \frac {N k_f}{2(b+1)} (s+\epsilon)^{-a} w^{b+1} \frac{1}{[1+ \bar{M}^2]^{\alpha}} \Big|_0^{R^N} \\[6pt]
\notag &\quad - \frac{N k_f}{2(b+1)} \int_0^{R^N}  \frac{d}{ds} \Big(\frac{(s+\epsilon)^{-a}}{\big[1+\bar{M}^2\big]^{\alpha}} \Big) w^{b+1}\,ds\\[6pt]
\notag &\geq -\frac{N k_f}{2(b+1)} \int_0^{R^N}  \frac{d}{ds} \Big(\frac{(s+\epsilon)^{-a}}{\big[1+\bar{M}^2\big]^{\alpha}} \Big) w^{b+1}\,ds\\[6pt]
 &= \frac{a N k_f}{2(b+1)} \int_0^{R^N}(s+\epsilon)^{-a-1} \frac{w^{b+1}}{\big[1+\bar{M}^2\big]^{\alpha}}\,ds.
\end{align*}
This leads to
\begin{align}\label{I21/2}
\frac{\mathcal{I}_{21}} 2  \geq \frac{a N k_f}{2(b+1)} \bar C\int_0^{R^N}(s+\epsilon)^{-a-1} w^{b+1} \,ds
\end{align}
 with $\bar C = \frac{1}{[1+ \bar{M}^2]^{\alpha}}$.
\vskip.2cm
Now, since $\frac 1 {\big[1+ s^{\frac 2 N -2} ( \frac {m(t)}N s - w) ^2\big]^{\alpha}} \leq 1$, we obtain
\begin{align}\label{I22}
\notag \mathcal{I}_{22}&= - \int_0^{R^N} s (s+\epsilon)^{-a} w^{b-1} w_s m(t) f\Bigl(s^{\frac 2 N -2}\Bigl(w-\frac{m(t)}N s\Bigr)^2\Bigr)\,ds\\
\notag&= - k_f \int_0^{R^N} s (s+\epsilon)^{-a} w^{b-1} w_s m(t) \frac 1 {\big[1+ s^{\frac 2 N -2} ( \frac {m(t)}N s - w) ^2\big]^{\alpha}} \,ds \\
&  \geq - k_f \int_0^{R^N} s (s+\epsilon)^{-a} w^{b-1} w_s m(t) \,ds \geq  - k_f \bar m |\Omega| \int_0^{R^N} s (s+\epsilon)^{-a} w^{b-1} w_s\, ds,
\end{align}
where in the last inequality we used \eqref{int u < bar m}.\\
Replacing \eqref{I_1}, \eqref{I21/2} and \eqref{I22} in  \eqref{w+eps} and integrating from $0$ to $t\in(0,T_{max})$ we arrive to
\begin{align}
\notag &\frac 1 b \int_0^{R^N} (s+\epsilon)^{-a} w^b(s,t) \,ds \geq \frac 1 b \int_0^{R^N} (s+\epsilon)^{-a} w_0^b(s)\, ds \\
\notag & - k_f \bar m |\Omega| \int_0^t  \int_0^{R^N} s(s+\epsilon)^{-a} w^{b-1} w_s \,ds d\tau\\
\notag & + \frac{a N k_f}{2(b+1)} \bar C \int_0^t \int_0^{R^N}(s+\epsilon)^{-a-1} w^{b+1} \,ds d\tau\\
\notag & +  \frac 1 2 N k_f  \bar C \int_0^t \int_0^{R^N} (s+\epsilon)^{-a} w^{b} w_s \,ds d\tau \\
\notag &+ N^2  (1-b)\int_0^t \int_0^{R^N}  s^{2- \frac 2N }(s+\epsilon)^{-a} w^{b-2} w^2_{s}\,ds d\tau\\
\notag & - 2N(N-1) \int_0^t  \int_0^{R^N} s^{1- \frac 2N} (s+ \epsilon)^{-a} w^{b-1} w_s\, ds d\tau\\
\notag
&- \mu N^{k-1}\int_0^t \int_0^{R^N} (s+\epsilon)^{-a} w^{b-1} \Bigl(\int_0^s w_s^k d\sigma \Bigr) \,ds d\tau.
\end{align}
Now, from the monotone convergence theorem, taking $\epsilon \searrow 0$ arrive at \eqref{s^-aw^b}
\end{proof}
Our aim is to  construct an integral inequality for $y(t)= \int_0^{R^N} s^{-a} w^b(s,t)\,ds, \ t\in(0,T_{max}) $ which ensure that $y(t)$ blows up in finite time inducing the chemotactic collapse of the solution of \eqref{sys}.\\ To this end, we estimate each term in \eqref{s^-aw^b}.\\
In \eqref{s^-aw^b} we assume $c_1:= \min \bigl\{ N^2 (1-b), \, \frac{a N k_f}{2(b+1)} \bar C \bigr\}$  to obtain  
\begin{align}\label{s^-aw^b bis}
\notag \frac 1 b \int_0^{R^N} s^{-a} w^b(s,t) \,ds 
&\geq \frac 1 b \int_0^{R^N} s^{-a} w_0^b(s)\, ds + c_1 \int_0^t \int_0^{R^N}s^{-a-1} w^{b+1} \,ds d\tau\\ 
\notag &\quad +   \frac 1 2 N k_f  \bar C \int_0^t \int_0^{R^N} s^{-a} w^{b} w_s \,ds d\tau \\
\notag &\quad+ c_1 \int_0^t \int_0^{R^N}  s^{2- \frac 2N  -a }w^{b-2} w^2_{s}\,ds d\tau\\
\notag &\quad 
- k_f \bar m |\Omega|^{-1} \int_0^t  \int_0^{R^N} s^{1-a} w^{b-1} w_s \,ds d\tau\\
\notag &\quad - 2N(N-1) \int_0^t  \int_0^{R^N} s^{1- \frac 2N -a} w^{b-1} w_s \,ds d\tau\\
\notag &\quad - \mu N^{k-1}\int_0^t \int_0^{R^N} s^{-a} w^{b-1} \Bigl(\int_0^s w_s^k \,d\sigma \Bigr)\, ds d\tau\\[2mm]
 & = H_1 + H_2 + H_3 + H_4 - H_5 -H_6 - H_7, \ \text{for all} \ t\in(0,T_{max}).
\end{align}

\vspace{2mm}

\begin{lemma}\label{H_5 H_6}
Let $H_5$ and $H_6$ defined as in \eqref{s^-aw^b bis}. If 
\begin{equation} \label{a}
0<a<\frac{N-2}N (b+1),
\end{equation}
then
\begin{align} \label{H_5}
H_5 & \leq  \frac 1 2 H_4   +  \frac 1 4 H_2+ c_4 t
\end{align}
\begin{align} \label{H_6}
H_6 \leq  \frac 1 2 H_4 + \frac 1 4 H_2  + c_6 t, \ \text{for all} \ t\in(0,T_{max}), 
\end{align}
with $c_4, \ c_6 >0$ and $H_2, \, H_4$ defined in \eqref{s^-aw^b bis}.
\end{lemma}
\begin{proof}
Using Young's inequality we obtain
\begin{align*} 
\notag H_5 &=k_f \bar m |\Omega|^{-1} \int_0^t  \int_0^{R^N} s^{1-a} w^{b-1} w_s \,ds d\tau \\
& \leq \frac{c_1} 2 \int_0^t \int_0^{R^N} s^{2- \frac 2 N  -a} w^{b-2} w^2_s \,ds d\tau 
  + c_2 \int_0^t \int_0^{R^N} s^{\frac 2 N -a} w^b \,ds d\tau \\
& \leq \frac {c_1} 2  \int_0^t \int_0^{R^N} s^{2- \frac 2 N  -a} w^{b-2} w^2_s \,ds d\tau \\
 &\quad +\frac {c_1} 4 \int_0^t \int_0^{R^N} s^{-a-1} w^{b+1}\, ds d\tau + c_3 \int_0^t \int_0^{R^N} s^{\frac 2 N - a + \frac{N+2}{N} b} \,ds d\tau.
\end{align*}
Since \eqref{a} holds we have 
$\frac 2 N -a + \frac{N+2}{N}  b >-1$, 
and for some $c_4 >0$ we obtain
\begin{align*}
H_5 & \leq  \frac 1 2 H_4   +  \frac 1 4 H_2+ c_4 t.
\end{align*}
To estimate $H_6$ we apply Young's inequality:
\begin{align*}
\notag H_6 & =  2N(N-1) \int_0^t  \int_0^{R^N} s^{1- \frac 2N -a} w^{b-1} w_s\, ds d\tau \\
\notag & \leq \frac{c_1} 2  \int_0^t  \int_0^{R^N} s^{2- \frac 2N -a} w^{b-2} w^2_s \,ds d\tau  + c_5 \int_0^t  \int_0^{R^N} s^{- \frac 2N -a} w^{b} \, ds d\tau  \\
\notag& \leq   \frac{c_1} 2 \int_0^t  \int_0^{R^N} s^{2- \frac 2N -a} w^{b-2} w^2_s\, ds d\tau + \frac{c_1}4     \int_0^t  \int_0^{R^N} s^{-a -1} w^{b+1} \,ds d\tau \\
\notag &\quad + \bar{c}_5 \int_0^t  \int_0^{R^N} s^{- \frac 2N -a + \frac {N-2}{N} b}\,ds d\tau \\
& \leq  \frac 1 2 H_4 + \frac 1 4 H_2  + c_6 t, \ \text{for all} \ t\in(0,T_{max}), 
\end{align*}
with $c_5, \ \bar{c}_5, \ c_6 >0$ and by \eqref{a}:   $- \frac 2N -a + \frac {N-2}{N} b >-1$.
\end{proof}

In order to estimate the term $H_7$ in \eqref{s^-aw^b bis} we prove the following lemma.
\begin{lemma}\label{lemma H_7}
Let $N\geq 3$, $R>0$ and $H_7$ be as in  \eqref{s^-aw^b bis}.\\
$\diamond$ If $k=2$ and $u_0$ satisfies \eqref{u_0}, then there exists a constant $\mu_0>0$ 
such that for all $\mu \in (0,\mu_0)$ one can find $a>1$ and $b \in (0,1)$ 
fulfilling \eqref{a} and
\begin{equation}\label{H_7 k=2}
 H_7 \leq  \frac 1 4 H_2.
\end{equation}
$\diamond$ If $k\in \bigl(1,\, \min\bigl\{2,1+\frac{(N-2)^2}{4}\bigr\}\bigr)$, then for all $\mu>0$ one can find $a, b \in (0,1)$ 
fulfilling \eqref{a} and
\begin{equation}\label{H_7 1<k<2}
 H_7 \leq \frac 1 4  H_2 + \bar c_2 t, \  \, \bar c_2 >0, \ \text{for all}\ t\in(0,T_{max}).
\end{equation}
\end{lemma}
\begin{proof}
By Fubini's theorem we obtain
\begin{align*}
 H_7&= \mu N^{k-1}\int_0^t \int_0^{R^N} s^{-a} w^{b-1} \Bigl(\int_0^s w_s^k d\sigma \Bigr) \,ds d\tau\\
&= \mu N^{k-1}\int_0^t  \int_0^{R^N} \Big(\int_{\sigma}^{R^N} s^{-a} w^{b-1} ds\Big)  w_s^k (\sigma) \,d\sigma d\tau.
\end{align*}
Since $b\in(0,1)$ and $w_s\geq 0$, then $w^{b-1}(s)$ decreases in $s$, we can write
\begin{align*}
H_7&\leq \mu N^{k-1}\int_0^t  \int_0^{R^N} \Big(\int_{\sigma}^{R^N} s^{-a} ds\Big) w^{b-1}(\sigma) w_s^k (\sigma)\, d\sigma d\tau\\
& = \frac 1 {1-a}  \mu N^{k-1}\int_0^t  \int_0^{R^N}  \big(R^{N(1-a)} - \sigma^{1-a} \big) w^{b-1}(\sigma) w_s^k (\sigma) \,d\sigma d\tau.
\end{align*}
In the case $k=2$, $a>1$ we neglect the negative term $-\frac{R^N}{a-1}$ and use \eqref{w_s<w/s} to obtain
\begin{align*}
H_7 &\leq \frac { \mu N} {a-1}  \int_0^t  \int_0^{R^N} s^{1-a} w^{b-1}(s) w_s^2 (s)\, ds d\tau\\
&\leq \frac { \mu N} {a-1}  \int_0^t  \int_0^{R^N} s^{-a-1} w^{b+1}\,  ds d\tau \leq \frac 1 4 H_2
\end{align*}
if $0<\mu \leq \frac {a-1} {4N} c_1$. 
We note that, from the definition of $c_1$, for some sufficiently small $\mu_0>0$,  
one can find $a>1$ and $b \in (0,1)$ 
fulfilling both \eqref{a} and $\mu_0 \leq \frac {a-1} {4N} c_1$.\\


If $k\in \bigl(1,\, \min\bigl\{2,1+\frac{(N-2)^2}{4}\bigr\}\bigr)$, $a\in (0,1)$ we neglect the negative term $- \frac 1 {1-a} \sigma^{1-a}$ and arrive to
\begin{align*}
H_7 \leq \frac { \mu N^{k-1}} {1-a} R^{N(1-a)} \int_0^t  \int_0^{R^N}  w^{b-1}(s) w_s^k (s)\, ds d\tau.
\end{align*}
We now fix $b=a \in \bigl(\sqrt{k-1},\, \min\bigl\{1,\frac{N-2}{2}\bigr\}\bigr)$ fulfilling \eqref{a}. 
This is possible in view of the choice of $k$, because 
\eqref{a} with $b=a$ is equivalent to $a<\frac{N-2}{2}$. 
Thus we see that $(a-1) \frac{a+1}{2-k}>-1$, and then \eqref{w_s<w/s} and Young's inequality lead to 
\begin{align*}
H_7& \leq \frac { \mu N^{k-1}} {1-a} R^{N(1-a)} \int_0^t  \int_0^{R^N} s^{-k} w^{k +a-1} ds d\tau\\
&\leq\int_0^t \Big[  \Big( \int_0^{R^N} s^{-a-1} w^{a+1} ds \Big)^{\frac{k+a-1}{a+1}} \Big(\int_0^{R^N} s^{(a-1) \frac{a+1}{2-k}}\,ds \Big)^{\frac{2-k}{a+1}}\Big] \,d\tau \\
&\leq \frac {c_1} 4  \int_0^t  \int_0^{R^N} s^{-a-1} w^{a+1}\,ds d\tau + \bar c_1 \int_0^t \int_0^{R^N} s^{(a-1) \frac{a+1}{2-k}}\,ds d\tau\\
&=\frac {c_1} 4  \int_0^t  \int_0^{R^N} s^{-a-1} w^{a+1}\,ds d\tau + \bar c_2 t, \ for  \ all \ t\in(0,T_{max}),
\end{align*}
with some $\bar c_2>0$. 
Thus we obtain \eqref{H_7 1<k<2} with $b=a$. 
\end{proof}

\vspace{2mm}

Taking into account of Lemmata \ref{lemma s^-aw^b}, \ref{H_5 H_6} and \ref{lemma H_7}, we derive an integral inequality for the functional $y(t)= \int_0^{R^N}  s^{-a} w^b(s) ds$.

\begin{lemma}\label{lemma integral ineq}
Suppose Lemma  \ref{H_5 H_6} and Lemma \ref{lemma H_7} hold. Let $N\geq 3$, $R>0$, $m_0>0,$ $\mu >0$ and $k\in (1,2]$. Then there exist $a>0$, $b\in (0,1)$, $\delta>0$ and $C>0$ such that if $u_0(r)$  is nonnegative in $B_R(0)\subset {\mathbb R}^N$ with $\frac 1 {|\Omega|} \int_{\Omega} u_0 = m_0$, for the corresponding solution $(u,v)$ of \eqref{sys} in $\Omega \times (0,T_{max})$ and $w$ defined in \eqref{w}, it holds  
\begin{align}\label{integral ineq}
\notag &\int_0^{R^N} s^{-a} w^b(s,t)\, ds \\
&\geq  \int_0^{R^N} s^{-a} w_0^b(s) \,ds 
 +\delta \int_0^t\Big( \int_0^{R^N}s^{-a} w^{b}(s,\tau)\, ds\Big)^{\frac{b+1}b}\, d\tau - C t 
\end{align}
for all $t\in(0,T_{max})$.

\end{lemma}
\begin{proof}
We analyse the two cases separately.\\
{\bf Case i)} Assume $k=2$, $1<a< \frac{N-2}N (b+1)$, $N\geq 5$, $0<\mu \leq \mu_0$. 
Thus $b\in(\frac 2 {N-2},1)$. \\
Substituting \eqref{H_5}, \eqref{H_6} and \eqref{H_7 k=2} in \eqref{s^-aw^b bis} and neglecting the positive term $H_3$, we see that
\begin{align*}
&\int_0^{R^N} s^{-a} w^b(s,t) \,ds \\
&\geq  \int_0^{R^N} s^{-a} w_0^b(s)\, ds 
 + \frac {b c_1} 4   \int_0^t \int_0^{R^N}s^{-a-1} w^{b+1} \,ds d\tau  - C t , \ \ \forall\, t\in(0,T_{max}).
\end{align*}

{\bf Case ii)} Assume $k\in \bigl(1,\,\min\bigl\{2,1+\frac{(N-2)^2}{4}\bigr\}\bigr)$, 
$b=a \in \bigl(\sqrt{k-1},\, \min\bigl\{1,\frac{N-2}{2}\bigr\}\bigr)$, $N\geq 3$, $\mu >0$.\\
Substituting \eqref{H_5}, \eqref{H_6} and \eqref{H_7 1<k<2} in \eqref{s^-aw^b bis} we obtain (with $b=a$)
\begin{align*}
\notag& \int_0^{R^N} s^{-a} w^b(s,t) \,ds \geq  \int_0^{R^N} s^{-a} w_0^b(s) \,ds + \frac {b c_1} 4   \int_0^t \int_0^{R^N}s^{-a-1} w^{b+1} \,ds d\tau \\
\notag &+ b  c_1 \int_0^t \int_0^{R^N} s^{-a} w^{b} w_s  \,ds d\tau - C t \\
&\geq  \int_0^{R^N} s^{-a} w_0^b(s)\, ds + \frac {b c_1} 4   \int_0^t \int_0^{R^N}s^{-a-1} w^{b+1} \,ds d\tau  - C t \ \ \forall\, t\in(0,T_{max}).
\end{align*}
In both cases i) and ii) we arrive at the following type inequality:
\begin{align} \label{s^-aw^b final}
\notag  &\int_0^{R^N} s^{-a} w^b(s,t) \,ds \\
&\geq  \int_0^{R^N} s^{-a} w_0^b(s) \,ds
 + \frac {b c_1} 4   \int_0^t \int_0^{R^N}s^{-a-1} w^{b+1}\, ds d\tau  - C t \ \ \forall\,t\in(0,T_{max}).
\end{align}
Now, by the H\"older inequality, we observe that
\begin{align*}
\int_0^{R^N} s^{-a} w^b \,ds 
&= \int_0^{R^N} s^{-a + \frac {b (a+1)}{b+1}} \big(s^{-a - 1} w^{b+1}\big)^{\frac b {b+1}}\,  ds \\
\notag &\leq \Big( \int_0^{R^N} s^{-a + b} \,ds \Big)^{\frac 1 {b+1}}  \Big( \int_0^{R^N} s^{-a - 1} w^{b+1} \,ds \Big)^{\frac b {b+1}} 
\end{align*}
from which we have
\begin{align}\label{s^-a-1 w^b+1}
& \int_0^{R^N} s^{-a - 1} w^{b+1} \,ds \geq \bar c_4 \Big(\int_0^{R^N} s^{-a} w^b \,ds\Big)^{\frac{b+1}b} 
\end{align}
with $\bar c_4=\Big( \frac {b+1 -a}{R^{N(b+1-a)}}\Big)^{\frac 1 b}$ and  $-a + b >-1$.\\
Replacing \eqref{s^-a-1 w^b+1} into \eqref{s^-aw^b final} we  arrive at \eqref{integral ineq} with $\delta = \frac 1 4 bc_1 \bar c_4$.
\end{proof}

\vskip.2cm

\begin{prth1.1}
{\rm By Lemma \ref{lemma integral ineq} with the aid of the Lemma \ref{lemma y(t)} and following the steps in the proof of Theorem 0.1 in \cite{W-2011},  we can conclude that $y(t)=\int_0^{R^N} s^{-a} w^b(s,t) ds$ blows up in finite time $T_{max} \leq \frac b {\delta \beta^{\frac 1 b}}$.\qed}
\end{prth1.1}


\section{Blow-up in $L^{p}$-norm}\label{BlowUp in L^p}
The aim of this section is to prove Theorem \ref{BULp}. To this end, first we prove the following lemma.
\begin{lem}\label{LemmaBoundedness u nabla v} 
Let $\Omega \subset\mathbb{R}^N,\ N\geq 3$ be a bounded and smooth domain. 
Let $(u,v)$ be a classical solution of system \eqref{sys}. 
If $\alpha$ satisfies \eqref{alpha} and if for some 
$p>\frac{N}{2} $
there exists 
$C>0$
such that 
\begin{align*}
\left\|u(\cdot,t)\right\|_{
L^{p}(\Omega)} 
\leq C, \quad \textrm{ for any } t \in (0,T_{max}),
\end{align*}
then, for some $\hat C>0$, 
\begin{align}\label{BoundednessU^infty}
\left\|u(\cdot,t)\right\|_{L^{\infty}(\Omega)} 
\leq {\hat C}, \quad \textrm{ for any } t \in (0,T_{max}).
\end{align}
\end{lem}

\begin{proof}
For any $t\in (0,T_{max})$, we set $t_0 := \max \{0, t-1\}$ and we consider the representation formula for $u$:
\begin{align*}
u(\cdot ,t) 
&= e^{(t-t_0)\Delta} u( \cdot, t_0) - k_f\int_{t_0}^{t} e^{(t-s)\Delta} \nabla \cdot \Big(u (\cdot, s)\frac{\nabla v(\cdot,s)}{(1+ |\nabla v(\cdot , t)|^2)^{\alpha}}\Big)\,ds \\
&\quad +  \int_{t_0}^t e^{(t-s) \Delta} \big( \lambda u (\cdot, s) - \mu u^k (\cdot, s) \big)\,ds =: u_1(\cdot,t)+u_2(\cdot,t)  + u_3(\cdot, t)
\end{align*}
and
\begin{align} \label{unorm}
 \| u(\cdot ,t) \|_{L^{\infty}} \leq \| u_1(\cdot,t) \|_{L^{\infty}(\Omega)} +\|u_2(\cdot,t)\|_{L^{\infty}(\Omega)}  + \| u_3(\cdot, t )\|_{L^{\infty}(\Omega)} .
\end{align}
We have
\begin{equation}\label{u_1}
\begin{split}                                                                                                                                                                                                                                                                                                                                                                                                                                                                                                                                                                                                                                                                                                                                                                                                                                                                                                                                                                                                                                                                                                                                                                                                                                                                                                                                                                                                                                                                                                                                                                                                                                                                                                                                                                                                                                                                                                                                                                                                                                                                                                                                                                                                                                                                                                                                                                                                                                                                                                                                                                                                                                                                                                                                                                                                                                                                                                                                                                                                                                                                     \lVert u_1 (\cdot,t) \rVert _{L^{\infty}(\Omega)} \leq \max \{\lVert u_0 \rVert_{L^{\infty}(\Omega)}, 2 \bar m k_1\} =:\tilde C_1,
\end{split}
\end{equation}
with $k_1>0$ and $\bar m$ defined in \eqref{bar m}. 
In fact, if $t\leq 1$, then $t_0=0$ and hence the maximum principle yields $u_1(\cdot, t) \leq \| u_0\|_{L^{\infty}(\Omega)}$. 
If $t>1$, then $t-t_0=1$ and from \eqref{bar m} and \eqref{etDeltaz} with ${\rm p}=\infty$ and $q=1$, we deduce that $\lVert u_1(\cdot,t)\rVert_{L^{\infty}(\Omega)} \leq k_1 [1+ (t-t_0)^{-\frac N 2}] e^{-\mu_1(t-t_0)} \lVert u(\cdot,t_0) \rVert_{L^1(\Omega)} \leq 2 \bar m k_1$.
\vskip.2cm
We next use \eqref{etDelta nablaz} with ${\rm p}=\infty$, which leads to
\begin{align}\label{u_2} 
& \|u_2 (\cdot, t) \|_{L^{\infty}(\Omega)} \\
\notag & \leq k_2 k_f  \int_{t_0}^{t} ( 1 + (t-s)^{-\frac 1 2 - \frac{N}{2q} }) e^{-\mu_1  (t-s)} \left\|u (\cdot, s)\frac{ \nabla v(\cdot,s)}{(1+|\nabla v|^2)^{\alpha}}  \right\|_{L^q(\Omega)} \,ds \\
\notag &\leq k_2 k_f \int_{t_0}^{t} ( 1 + (t-s)^{-\frac 1 2 - \frac{N}{2q} }) e^{-\mu_1  (t-s)} \|u (\cdot, s) |\nabla v|^{1- 2\alpha}  \|_{L^q(\Omega)} \,ds,
\end{align}
because $\frac {|\nabla v|}{(1+|\nabla v|^2)^{\alpha}}\leq |\nabla v|^{1-2\alpha}$.\\
Here, we may assume that 
$\frac N 2<p<N$,
and then 
we can fix $N< q < \frac{N p}{N-p}=p^*$. Since $2 \alpha < 1$, by H$\ddot{{\rm o}}$lder's inequality, we can estimate the last term in \eqref{u_2} as
\begin{align*}
& \|u (\cdot, s)|\nabla v(\cdot,s)|^{1-2\alpha}  \|_{L^q(\Omega)}\\
&\le 
\|u (\cdot, s)\|_{L^{\frac{q}{2\alpha}}(\Omega)}\|\nabla v(\cdot,s)  \|^{1-2\alpha}_{L^q(\Omega)}\\
&\le C_2\|u (\cdot, s)\|_{L^{\frac q {2\alpha}}(\Omega)} \|\nabla v(\cdot,s)\|^{1-2\alpha}_{L^{p^*}(\Omega)}\quad 
{\rm for\ all}\ s \in (0, \Tmax),
\end{align*}
for some $C_2>0$. 
The Sobolev embedding theorem and elliptic regularity theory 
for the second equation in \eqref{sys} tell us that 
$\|v(\cdot,s)\|_{W^{1,p^*}(\Omega)}
\leq C_3\|v(\cdot,s)\|_{W^{2,p}(\Omega)}
\le C_4$ with some $C_3, C_4>0$. 
Thus again by H$\ddot{{\rm o}}$lder's inequality, the definition of $\bar m $ and interpolation's inequality, we obtain
\begin{align*}
\|u (\cdot, s)|\nabla v(\cdot,s)|^{1-2\alpha}  \|_{L^q(\Omega)} 
&\le 
C_5\|u (\cdot, s)\|_{L^{\frac q{2\alpha}}(\Omega)}\\
&\le C_{5}\|u (\cdot, s)\|^{\theta}_{L^{\infty}(\Omega)} \|u (\cdot, s)\|^{1-\theta}_{L^1(\Omega)}\\
&\le C_{6}\|u (\cdot, s)\|^{\theta}_{L^{\infty}(\Omega)}\quad 
{\rm for\ all}\ s \in (0, \Tmax),
\end{align*}
with $\theta := 1 - \frac{2\alpha}{q} \in (0,1)$, $C_5:=C_2C_4$ and $C_{6}:=C_5 {\bar m}^{1-\theta}$. 
Hence, combining this estimate and \eqref{u_2}, we infer
\begin{align*}
\|u_2 (\cdot, t) \|_{L^{\infty}(\Omega)}
\leq C_{6}k_2  \int_{t_0}^{t} ( 1 +  (t-s)^{-\frac 1 2 - \frac{N}{2q} }) e^{-\mu_1  (t-s)} \|u (\cdot, s)\|_{L^{\infty}(\Omega)}^{ \theta}\,ds.
\end{align*}
Now fix any $T \in (0, T_{max})$. 
Then, since $t-t_0\leq 1$, we have
\begin{align}\label{u_21}
\|u_2 (\cdot, t) \|_{L^{\infty}(\Omega)}
&\leq C_{6}k_2   \int_{t_0}^{t} ( 1 +  (t-s)^{-\frac 1 2 - \frac{N}{2q} } e^{-\mu_1  (t-s)}) \,ds \cdot \sup_{t \in [0, T]} \|u (\cdot, t)\|_{L^{\infty}(\Omega)}^{ \theta}\notag\\[6pt]
&\leq C_{7}\sup_{t \in [0, T]} \|u (\cdot, t)\|_{L^{\infty}(\Omega)}^{ \theta},
\end{align}
where $C_{7}:=C_{6}k_2\bigl(1+\mu_1^{\frac{N}{2q}-\frac{1}{2}}\int_0^\infty r^{-\frac 1 2 - \frac{N}{2q}} e^{-r}\,dr\bigr)>0$ is finite, 
because $\frac{1}{2}+\frac{N}{2q}<1$  (i.e., $q>N$). \\


Now we prove that there exists a constant $c_8$ such that $\| u_3\| \leq c_8.$ In fact we observe that $g(u)= \lambda u - \mu u^k \leq g(\tilde u) := c_8$, with $\tilde u= \big( \frac {\lambda}{\mu} \big)^{\frac 1 {k-1}}$
\begin{align} \label{u_3}
\notag \|u_3(\cdot,t) \|_{L^{\infty}(\Omega)} 
&= \int_{t_0}^t \| e^{(t-s) \Delta} \big[ \lambda u (\cdot, s) - \mu u^k (\cdot, s) \big]\|_{L^{\infty}(\Omega)}\,ds \\[6pt]
&\leq  \int_{t_0}^t \| c_8  e^{(t-s) \Delta} \|_ {L^{\infty}(\Omega)} \,ds \leq c_8 (t-t_0) \leq c_8.
\end{align}

\vskip.2cm
Plugging \eqref{u_1}, \eqref{u_21} and \eqref{u_3}
into \eqref{unorm},	we see that
\begin{align}\label{u final} 
\| u(\cdot, t) \|_{L^{\infty}} \le C_1+C_{7} \sup_{t \in [0, T]} \|u (\cdot, t)\|_{L^{\infty}(\Omega)}^{ \theta},
\end{align}
with $C_1=\tilde C_1 + c_8$.

The inequality \eqref{u final} implies
\begin{align*}
\sup_{t \in [0, T]}\|u(\cdot, t)\|_{L^\infty(\Omega)} 
&\le C_{1}+C_{7} \Big(\sup_{t \in [0, T]} \|u (\cdot, t)\|_{L^{\infty}(\Omega)}\Big)^{\theta}\quad {\rm for\ all}\ T \in (0, T_{max}).
\end{align*} 
From this inequality with $\theta \in (0,1)$, we arrive at \eqref{BoundednessU^infty}.
\end{proof}

\vspace{2mm}

\begin{prth1.2}
\rm{Since Theorem~\ref{BULinfty} holds, the unique local classical solution of \eqref{sys} blows up at $t=T_{max}$ in the sense of \eqref{blowupinfty}, that is,  
\[
\limsup_{t \nearrow \Tmax}\| u(\cdot , t) \|_{L^{\infty}(\Omega)}= \infty.
\]
We prove that it blows up also in $L^p$-norm by contradiction.\\ 
In fact, if one supposes that there exist $p> \frac {N}{2} $ and $C>0$ 
such that 
\begin{align*}
\|u(\cdot, t)\|_{L^{p}(\Omega)} \leq C,\quad 
{\rm for\ all}\ t \in (0, T_{max}),
\end{align*} 
then, from Lemma \ref{LemmaBoundedness u nabla v}, it would exist 
$\hat C>0$ such that 
\begin{equation*} 
\lVert u(\cdot,t)\rVert _{L^\infty(\Omega)}\leq \hat C,\quad
{\rm for\ all}\ t \in (0, T_{max}),
\end{equation*}
which contradics \eqref{blowupinfty}. Thus, 
if $u$ blows up in $L^{\infty}$-norm, 
then $u$ blows up also in $L^p$-norm 
for all $p>\frac{N}{2}$}. \qed
\end{prth1.2}


\section{Lower bound of the blow-up time $\Tmax$} \label{lower bound}
Throughout this section we assume that Theorem \ref{BULp} holds.\\

We want to obtain a safe interval of existence of the solution of \eqref{sys}  $[0,T]$, with $T$ a lower bound of the blow-up time $T_{max}$. To this end, first we construct a first order differential inequality for $\Psi$ defined in \eqref{Psi} and by integration we get the lower bound.

\begin{prth1.3} 
{\rm
By differentiating  \eqref{Psi} we have
\begin{align}
 \label{Psi'}
\Psi'(t)
&= \int_\Omega u^{p -1} \Delta u \,dx
- \int_\Omega u^{p -1}\nabla\cdot (u \nabla v f(|\nabla v|^2 ) \, dx + \lambda \int_{\Omega} u^p \,dx - \mu \int_{\Omega} u^{p+k-1} \,dx\\
& \notag=:\mathcal  J_1+  \mathcal J_2 + \mathcal J_3+\mathcal J_4
\end{align}
with
\begin{align}
 \label{J1}
&\mathcal  J_1 = \int_\Omega u^{p -1} \Delta u \,dx\\ 
\notag&=  \int_\Omega \nabla \cdot\big( u^{p -1}\nabla u\big) \,dx - (p-1) \int_{\Omega}  u^{p-2} | \nabla u|^2 \,dx\\ 
\notag&  =  - \frac{4(p-1)}{p^2} \int_{\Omega}  | \nabla u^{\frac p 2}|^2 \,dx.
\end{align}
In the second term of \eqref{Psi'}, integrating by parts and using the boundary conditions in \eqref{sys}, for all $t\in [0,T_{max})$  we obtain
\begin{align}\label{J2}
&\mathcal  J_2 =- \int_{\Omega} u^{p -1}\nabla\cdot (u \nabla v f(|\nabla v|^2 ) \, dx \\
\notag&= (p-1) \int_{\Omega}  f(|\nabla v|^2)    u^{p -1} \nabla u \cdot \nabla v\,dx\\
\notag&=\frac{p-1}{p}  \int_{\Omega} \nabla u^p \cdot \nabla v f (|\nabla v|^2)\, dx \\
\notag&= -\frac{p-1}{p} \int_{\Omega} u^p \nabla \cdot[ \nabla v f (|\nabla v|^2)] \,dx\\
\notag&= -\frac{p-1}{p} \int_{\Omega} u^p [\Delta v  f (|\nabla v|^2)] \,dx\\
\notag&- \frac{p-1}{p} \int_{\Omega} u^p f'(|\nabla v|^2) \nabla v \cdot \nabla (|\nabla v|^2) \,dx.
\end{align}
Using the second equation of \eqref{sys} and taking into account that $f(\xi)= k_f(1+\xi)^{-\alpha}$, $  f'(\xi )=  -\alpha k_f(1+\xi)^{-\alpha-1}$ in \eqref{J2}, we have
\begin{align}
 \label{J2 bis}
\mathcal J_2  &= - k_f\frac{p-1}{p} \int_{\Omega} u^p \frac{m(t) - u}{(1+ |\nabla v|^2)^{\alpha}}   \,dx \\
\notag &\quad + \alpha k_f \frac{p-1}{p} \int_{\Omega}u^p \frac{ \nabla v \cdot \nabla (|\nabla v|^2) }{(1 + |\nabla v|^2)^{\alpha + 1}}\,dx\\[2pt]
\notag & \leq k_f\frac{p-1}{p} \int_{\Omega}  u^{p +1}  \,dx  +  \alpha k_f  \frac{p-1}{p} \int_{\Omega}  u^{p} \frac{\nabla v \cdot \nabla (|\nabla v|^2)}{(1+ |\nabla v|^2)^{\alpha +1}} \,dx,
\end{align}
where we dropped the negative term $- k_f\frac{p-1}{p} \int_{\Omega} u^p \frac{m(t) }{(1+ |\nabla v|^2)^{\alpha}}   dx$ and used the inequality $\frac{1}{(1+ |\nabla v|^2)^{\alpha} }\leq 1$ as $\alpha>0$. \\
In order to estimate the second term of \eqref{J2 bis} we recall the radially symmetric setting to obtain (with $\omega_N$ the surface area of the unit sphere in $N$ dimension)
\begin{align*}
\int_{\Omega}  u^{p} \frac{\nabla v \cdot \nabla (|\nabla v|^2)}{(1+ |\nabla v|^2)^{\alpha +1}} \,dx 
&= \omega_N \int_0^R u^p \frac{Nv_r(v^2_r)_r}{(1+ v^2_r)^{\alpha +1}} r^{N-1} \,dr\\
& =2N\omega_N \int_0^R u^p \frac{v^2_r v_{rr}}{(1+ v^2_r)^{\alpha +1}} r^{N-1} \,dr,
\end{align*}
which together with $v_{rr}= \frac{m(t)}{N} - u + \frac{N-1} {r^N} \int_0^r \rho^{N-1} u \ d  \rho $ implies
\begin{align}
 \label{J2 4}
&\int_{\Omega}  u^{p} \frac{\nabla v \cdot \nabla (|\nabla v|^2)}{(1+ |\nabla v|^2)^{\alpha +1}} \,dx\\
\notag&= 2m(t) \omega_N \int_0^R u^p \frac{v^2_r }{(1+ v^2_r)^{\alpha +1}} r^{N-1} \,dr\\
\notag&\quad - 2N\omega_N \int_0^R u^{p+1}  \frac{v^2_r}{(1+ v^2_r)^{\alpha +1}} r^{N-1}\, dr \\
\notag & \quad + 2N(N-1) \omega_N \int_0^R u^p \frac{v^2_r }{(1+ v^2_r)^{\alpha +1}}  \frac 1 r \Big(\int_0^r \rho^{N-1} u \,d \rho \Big) \,dr \\
\notag& \leq 2\frac{\bar{m}}{|\Omega|}\omega_N \int_0^R u^p r^{N-1} \,dr  + 2N(N-1) \omega_N \int_0^R u^p  \frac 1 r \Big(\int_0^r \rho^{N-1} u \,d \rho \Big) \,dr,
\end{align}
where we used \eqref{int u < bar m}, we dropped the negative term $- 2N\omega_N \int_0^R u^{p+1}  \frac{v^2_r}{(1+ v^2_r)^{\alpha +1}} r^{N-1}\, dr$ and finally we used the inequality $\frac{v^2_r}{(1+ v^2_r)^{\alpha + 1} }\leq 1.$ \\
In the second term of \eqref {J2 4}, H$\ddot{{\rm o}}$lder's inequality yelds that for all $\epsilon >0$ there exists $c= c(\epsilon, N, p)$ such that
\begin{align}
 \label{J2 5}
& \omega_N \int_0^R u^p  \frac 1 r  \Big(\int_0^r \rho^{N-1} u \,d \rho \Big) dr \\
 \notag& \leq  \omega_N \int_0^R u^p  \frac 1 r  \Big(\int_0^r \rho^{N-1} \,d \rho \Big)^{\frac {p}{p+1}}\Big( \int_0^r u^{p+1} \rho^{N-1} \,d \rho\Big)^{\frac 1{p+1}} dr\\
 \notag& \leq \Big(\frac 1 {N}\Big)^{\frac p {p+1}} \Big(\int_{\Omega} u^{p+1} \,dx \Big)^{\frac 1 {p+1}} \omega^{\frac p {p+1}}_N \int_0^R u^p r^{\frac {Np}{p+1}-1}\, dr \\
 \notag& \leq \Big( \! \frac {1} N \! \Big)^{\frac p{p+1}} \! \Big( \! \int_{\Omega} u^{p+1} \,dx \! \Big)^{ \frac 1 {p+1}} \omega_N^{\frac p {p+1}}\! \Big(\! \int_0^R u^{p+1+\epsilon} r^{N-1} \,dr\!  \Big)^{\frac p {p+1+\epsilon}} \Big(\! \int_0^R  r^{\frac{\epsilon Np}{ p+1} - 1} dr\!  \Big)^{\frac{1+\epsilon}{p+1+\epsilon}}\\
\notag& = c \Big( \int_{\Omega} u^{p+1} \,dx  \Big)^{\frac 1 {p+1} } \Big(\int_{\Omega} u^{p+1+\epsilon} \,dx\Big)^{\frac{p}{p+1+ \epsilon}}.
\end{align}
Combining \eqref{J2 5} and \eqref{J2 4} with \eqref{J2 bis} we obtain
\begin{align}
 \label{I2 6}
 \mathcal J_2  &\leq 2 \alpha \frac{\bar{m}}{|\Omega|}  k_f\frac{p-1}{p}   \int_{\Omega} u^p\, dx  + k_f \frac{p-1}{p} \int_{\Omega}  u^{p +1}  \,dx \\
 \notag&\quad + 2 \alpha N(N-1) c k_f \;  \frac{p-1}{p} \Big( \int_{\Omega} u^{p+\,1}\, dx \Big)^{\frac 1 {p+1}}\Big( \int_{\Omega} u^{p+1+\epsilon} \,dx\Big)^{\frac {p}{p+1+\epsilon}} \\
 \notag& \leq \frac{ c_1}{p}  \int_{\Omega} u^p \,dx  + c_2  \int_{\Omega} u^{p +1} \, dx + c_3\Big( \int_{\Omega} u^{p+1+\epsilon}\, dx\Big)^{\frac {p+1}{p+1+\epsilon}}
 \end{align}
where, in the last term, we used Young's inequality 
 with $c_1 = 2 \alpha\frac{\bar{m}}{|\Omega|} k_f(p-1)  , \ \  c_2 =  k_f\frac{p-1}{p}+ 2 \alpha N(N-1) c k_f\,  \frac{p-1}{p(p+1)}, \ \ c_3= 2 \alpha N(N-1) c  k_f\;  \frac{p-1}{p+1} $. \\
Thanks to the Gagliardo--Nirenberg inequality \eqref{GN ineq}, with $\mathsf{p}= 2\frac{p+1}{p}, \   \mathsf{r} = \mathsf{q}= \mathsf{s}=2, \ a= \theta_0 := \frac {N}{2(p+1)} \in(0,1)$ for all $p> \frac N 2$, we see that
\begin{align}
 \label{u^p+1}
 & \int_{\Omega} u^{p+1} \,dx = \| u^{\frac p 2}\|_{L^{2\frac{p+1}{p}}(\Omega)}^{2\frac{p+1}{p}} \\
&  \notag  \leq C_{GN} \| \nabla u^{\frac p 2} \|_{L^2(\Omega)}^{ 2\frac{p+1}{p}\theta_0}  \| u^{\frac p 2}\|_{L^2(\Omega)}^{ 2\frac{p+1}{p}(1-\theta_0)}+ C_{GN}   \| u^{\frac p 2}\|_{L^2(\Omega)}^{ 2\frac{p+1}{p}}\\
&  \notag = C_{GN}  \Big( \int_{\Omega} |\nabla u^{\frac p 2}|^2 \,dx\Big) ^{ \frac{N}{2p}}  \Big(  \int_{\Omega} u^{ p}\, dx\Big) ^{ \frac{2(p+1) - N}{2p}}+  C_{GN}   \Big(\int_{\Omega} u^{p} \,dx\Big)^{\frac{p+1}p}.
  \end{align}
Applying Young's inequality at the first term of \eqref{u^p+1} we have
\begin{align}
 \label{u^p+1 bis}
 &\int_{\Omega} u^{p+1} \,dx  \leq  \frac{N}{2p} \epsilon_1 C_{GN}  \int_{\Omega}  |\nabla u^{\frac p 2}|^2  \,dx\\
 & \notag + C_{GN}  \frac{2p-N}{2p \epsilon_1^{\frac N{2p-N}}} \Big(\int_{\Omega}  u^p \,dx\Big)^{\frac {2(p+1)-N}{2p-N}} + C_{GN}  \Big(\int_{\Omega} u^{p} \,dx\Big)^{\frac{p+1}p}
 \end{align}
with $\epsilon_1 >0$ to be choose later on,  and also
\begin{align}\label{u^p+1+ epsilon}
&\Big(\int_{\Omega} u^{p+1+\epsilon}\, dx \Big)^{\frac {p+1}{p+1+\epsilon}} = \| u^{\frac p 2}\|^{2\frac{p+1}p}_{L^2\frac{p+1+\epsilon}{p}(\Omega)} \\
& \notag \leq  C_{GN} \| \nabla u^{\frac p 2} \|_{L^2(\Omega)}^{ 2\frac{p+1}{p}\theta_{\epsilon}}  \| u^{\frac p 2}\|_{L^2(\Omega)}^{ 2\frac{p+1}{p}(1-\theta_{\epsilon})}+ C_{GN}   \| u^{\frac p 2}\|_{L^2(\Omega)}^{ 2\frac{p+1}{p}}\\
& \notag  =  C_{GN}  \Big(  \int_{\Omega}  |\nabla u^{\frac p 2}|^2 \, dx \Big)^{\frac{p+1}{p}\theta_{\epsilon}}  \Big(   \int_{\Omega} u^p \,dx \Big)^{ \frac{p+1 }{p}(1-\theta_{\epsilon})}
 +  C_{GN}  \Big(\int_{\Omega} u^{p} \,dx\Big)^{\frac{p+1}p},
\end{align}
\vskip.2cm
with $\mathsf{p}= 2\frac {p+1}{p},\  \mathsf{r} = \mathsf{q}= \mathsf{s}=2, \ a= \theta_{\epsilon} := \frac {N(1+ \epsilon)}{2(p+1+\epsilon)} \in(0,1)$ for all $p> \frac N 2$ and sufficiently small $\epsilon >0$.\\
Now, in the first term of \eqref{u^p+1+ epsilon}, we apply Young's inequality to obtain
\begin{align}
 \label{u^p+1+ epsilon bis}
&  \Big(\int_{\Omega} u^{p+1+\epsilon} \,dx \Big)^{\frac {p+1}{p+1+\epsilon}}  \\
 \notag  & \leq  c_4 \int_{\Omega}  |\nabla u^{\frac p 2}|^2 \, dx + c_5\Big(\int_{\Omega}  u^p \,dx\Big)^{\gamma} +  C_{GN}  \Big(\int_{\Omega} u^{p}\, dx\Big)^{\frac{p+1}p},
\end{align}
\vskip.2cm
with 
\begin{align*}
&c_4:= \frac{N(1+\epsilon)(p+1) }{2p(p+1 + \epsilon)} \epsilon_2 C_{GN} , \\
&c_5:= C_{GN} \Big(\frac{2p (p+1+ \epsilon)- N(p+1) (1+\epsilon)}{2p(p+1+\epsilon)} \Big) \epsilon_2^{\frac{N(1+\epsilon)}{2(p+1+\epsilon) -N(1+\epsilon)}},\\
&\gamma :=\frac{2(p+1) - \frac{N(p+1)(1+\epsilon)}{p+1+\epsilon} }{2p-\frac{N(1+\epsilon)(p+1)}{p+1+\epsilon}}, \ \ \epsilon_2 >0.
\end{align*}
Note that we can fix $\epsilon>0$ such that $2p -N(1+\epsilon) >0$.\\
Plugging \eqref{u^p+1 bis} and \eqref{u^p+1+ epsilon bis} into \eqref{I2 6} leads to
\begin{align}
 \label{J2 final}
& \mathcal J_2  \leq C \int_{\Omega}  |\nabla u^{\frac p 2}|^2 \, dx + \frac{c_1}{p} \int_{\Omega} u^p \,dx + C_{GN}  \Big(\int_{\Omega}  u^p \,dx\Big)^{\frac {p+1}{p}} \\
&  \notag  +  \tilde c_1\Big(\int_{\Omega}  u^p dx\Big)^{\frac {2(p+1)-N}{2p-N}}+c_5 \Big(\int_{\Omega} u^{p} \,dx\Big)^{\gamma}  
 \end{align}
with $C:= c_3\cdot c_4, \ \ \tilde c_1:= C_{GN}  \frac{2p-N}{2p \epsilon_1^{\frac N{2p-N}}}c_2, \ \epsilon_1>0.$\\ 
Also we note that
\begin{equation}\label{J_3}
\mathcal J_3 = \lambda \int_{\Omega} u^p \,dx =B_1\Psi, \ \ \  B_1= \lambda p.
\end{equation}


Finally, combining \eqref{J2 final} with \eqref{Psi'} and \eqref{J1}, \eqref{J_3}, neglecting the negative term $\mathcal J_4$ and choosing $\epsilon_2$ such that the term containing $\int_{\Omega} |\nabla u^{\frac p 2}|^2 dx$ vanishes,  we have
\begin{equation}\label {Psi' final}
\Psi' \leq  B_1 \Psi + B_2 \Psi^{\frac {p+1}{p}} +B_3 \Psi^{\frac {2(p+1)-N}{2p-N}} + B_4 \Psi^{ \gamma},
\end{equation}
with $B_2:= p^{\frac 1 p} [p C_{GN} + c_1]$, $ B_3:= \tilde c_1 p^{\frac {2(p+1)-N}{2p-N}}$ and $B_4:= c_5 p^{\gamma}$ .\\
Integrating \eqref{Psi' final} from $0$ to $T_{max}$, we arrive at the desired lower bound \eqref{lower Tmax in Lp} with $\gamma_1 := \frac {p+1}{p}, \ \gamma_2:= \frac {2(p+1)-N}{2p-N}$.
}\qed 
\end{prth1.3}

\begin{prcor1.1}

\rm{ We reduce \eqref{Psi' final} so as to have an explicit expression of the lower bound $T$ of $T_{max}$. In fact, since $\Psi(t)$ blows up at time $T_{max}$, there exists a time $t_1 \in (0, T_{max})$ such that $\Psi(t) \geq \Psi_0$ for all $t\in (t_1, T_{max})$.
 Thus, taking into account that 
 \begin{align*}
1<  \gamma_1 <\gamma_2 <\gamma
\end{align*}
we have
 \begin{align} \label{Psi^sigma}
& \Psi  \leq \Psi^{\gamma} \Psi_0^{1-\gamma}, \\
& \notag \Psi^{\gamma_i}  \leq \Psi^{\gamma} \Psi_0^{\gamma_i - \gamma} ,\ \ \ i=1,2.
\end{align}
From \eqref{Psi' final} and \eqref{Psi^sigma} we arrive at
\begin{align}\label {Psi' final new}
\Psi' \leq \mathcal{A} \Psi^{\gamma}, \ \  \forall\,t\in (t_1, T_{max}),
\end{align}
with $\mathcal{A}:= B_1 \Psi_0^{1-\gamma} + B_2 \Psi_0^{ \gamma_1-\gamma} + B_3   \Psi_0^{\gamma_2 - \gamma} + B_4 $, and $\Psi_0$ in \eqref{Psi}.

Integrating \eqref{Psi' final new} from $t=0$ to $t= T_{max}$, we obtain 
\begin{align}\label {lower}
\frac 1 { (\gamma -1) \Psi_0^{\gamma-1}}= \int_{\Psi_0}^{\infty} \frac {d\eta}{\eta^{\gamma}} \leq \mathcal{A} \int_{t_1}^{T_{max}} d\tau \leq \mathcal{A} \int_{0}^{T_{max}} d \tau =\mathcal{A} T_{max}.
\end{align}
We conclude, by \eqref{lower}, that the solution of \eqref{sys} is bounded in $[0,T]$ with $T:= \frac 1 { \mathcal{A} (\gamma -1) \Psi_0^{\gamma-1}}.$
}\qed 
\end{prcor1.1}


\section{Global existence and boundedness}\label{boundedness}

The aim of this section is to prove Theorem \ref{GEB}. 
The proof is divided into two cases.

\subsection{Case 1. $\alpha>\frac{N-2}{2(N-1)}$ and $k>1$}
As in the proof of Lemma \ref{LemmaBoundedness u nabla v}, for any $t \in (0,T_{max})$, we set 
$t_0:=\max\{0, t-1\}$. From the representation formula for $u$ we can write
\begin{align*}
      u(\cdot,t) 
   &= e^{(t-t_0)\Delta}u(\cdot,t_0) -\int_{t_0}^t e^{(t-s)\Delta}\nabla\cdot\left[u(\cdot, s)f(|\nabla v(\cdot,s)|^2)\nabla v(\cdot,s)\right]\,ds
\\
   &\quad\ +\int_{t_0}^t e^{(t-s)\Delta} g(u)\,ds
    =: u_1(\cdot,t)+u_2(\cdot,t)+u_3(\cdot,t). 
\end{align*}
In view of \eqref{unorm} and \eqref{u_1} as well as \eqref{u final}  we have
\begin{align}\label{tomomi01}
      \left\|u(\cdot,t)\right\|_{L^\infty(\Omega)} 
   &\le c_1+ \left\|u_2(\cdot,t)\right\|_{L^\infty(\Omega)}.
\end{align}
Since the condition $\alpha>\frac{N-2}{2(N-1)}$ implies that 
$(1-2\alpha)N<\frac{N}{N-1}$, 
we can take $q \in \left[1,\frac{N}{N-1}\right)$ such that $q>(1-2\alpha)N$, 
and hence we pick $r>N$ satisfying $q>(1-2\alpha)r$. Then 
we see from the second equation in \eqref{sys} with mass estimate \eqref{int u < bar m} that 
\[
     \sup_{t \in (0,T_{max})} \|\nabla v(\cdot,t)\|_{L^q(\Omega)} \le c_2.
\]
Using \eqref{etDelta nablaz} with ${\rm p}=\infty$ and $q=r$ as in \eqref{u_2}, we deduce from the H\"older inequality that 
\begin{align*}
   & \left\|u_2(\cdot,t)\right\|_{L^\infty(\Omega)} \\
   &\le c_3\int_{t_0}^t (1+(t-s)^{-\frac{1}{2}-\frac{N}{2r}})e^{-\mu_1(t-s)}\|u(\cdot,s)|\nabla v(\cdot,s)|^{1-2\alpha}\|_{L^r(\Omega)}\,ds \\
   &\le c_3\int_{t_0}^t (1+(t-s)^{-\frac{1}{2}-\frac{N}{2r}})e^{-\mu_1(t-s)}\|u(\cdot,s)\|_{L^{\frac{qr}{q-(1-2\alpha)r}}(\Omega)}
\|\nabla v(\cdot,s)\|_{L^q(\Omega)}^{1-2\alpha} \,ds.
\end{align*}
Putting $a:=1-\frac{q-(1-2\alpha)r}{qr} \in (0,1)$ and recalling \eqref{int u < bar m} again, 
we note that 
\[
\|u(\cdot,s)\|_{L^{\frac{qr}{q-(1-2\alpha)r}}(\Omega)} \le \left\|u(\cdot,s)\right\|_{L^\infty(\Omega)}^a\left\|u(\cdot,s)\right\|_{L^1(\Omega)}^{1-a} \le c_4\left\|u(\cdot,s)\right\|_{L^\infty(\Omega)}^a,
\]
and hence,  
\begin{align*}
    \left\|u_2(\cdot,t)\right\|_{L^\infty(\Omega)}
   &\le c_2 c_3 c_4 \int_{t_0}^t (1+(t-s)^{-\frac{1}{2}-\frac{N}{2r}})e^{-\mu_1(t-s)}
     \|u(\cdot,s)\|_{L^{\infty}(\Omega)}^a \,ds.
\end{align*}
This together with \eqref{tomomi01} implies that for any $T \in (0,T_{max})$, 
\begin{align*}
   & \sup_{t \in (0,T)}\left\|u_2(\cdot,t)\right\|_{L^\infty(\Omega)} \\
   &\le c_1 + c_2 c_3 c_4 \sup_{t \in (0,T)}\|u(\cdot,t)\|_{L^{\infty}(\Omega)}^a\int_{t_0}^t (1+(t-s)^{-\frac{1}{2}-\frac{N}{2r}})e^{-\mu_1(t-s)}
      \,ds
\\
   &\le c_1 + c_5 \Bigl(\sup_{t \in (0,T)}\|u(\cdot,t)\|_{L^{\infty}(\Omega)}\Bigr)^a 
\end{align*}
and thereby we conclude that $T_{max}=\infty$ and $\|u(\cdot,t)\|_{L^{\infty}(\Omega)} 
\le c_6$ for all $t>0$. 
\hfill \qedsymbol

\subsection{Case 2. $\alpha>0$ and $k>2$ in the radial setting}
We will derive a uniform estimate for $\Psi(t):=\frac{1}{p}\|u(\cdot,t)\|_{L^p(\Omega)}^p$ 
defined in \eqref{Psi}.  
As in the proof of Theorem \ref{LB} in Section \ref{lower bound}, we have 
\begin{align*}
   \Psi'(t)
  &=\int_\Omega u^{p-1}\Delta u\,dx -\int_\Omega u^{p-1}\nabla\cdot(u f(|\nabla v|^2\nabla v ))\,dx
    +\lambda \int_\Omega u^p\,dx -\mu\int_\Omega u^{p+k-1}\,dx
\\
  &=: \mathcal{J}_1+\mathcal{J}_2+\mathcal{J}_3+\mathcal{J}_4.
\end{align*}
In view of \eqref{J1}, \eqref{J2 final} and \eqref{J_3} we see that
\begin{align*}
    \mathcal{J}_1&=-\frac{4(p-1)}{p^2}\int_\Omega |\nabla u^\frac{p}{2}|^2\,dx,
\\
    \mathcal{J}_2&\le c_1\ep_2 \int_\Omega |\nabla u^\frac{p}{2}|^2\,dx
     +c_2 \Psi(t) +c_3 \Psi^\frac{p+1}{p}(t)
     +c_4 \Psi^\frac{2(p+1)-N}{2p-N} (t)
     +c_5 \Psi^{\gamma}(t),
\\
    \mathcal{J}_3&=\lambda p \Psi(t)
\end{align*}
and the H\"older inequality yields
\[
    \mathcal{J}_4 \le - c_6 \Psi^\frac{p+k-1}{p}(t). 
\]
Choosing $\ep_2$ such that the term containing $\int_\Omega |\nabla u^\frac{p}{2}|^2\,dx$ vanishes 
and noting that $k>2$ implies $\frac{p+1}{p} \in (1,\frac{p+k-1}{p})$ and 
\[
  \frac{2(p+1)-N}{2p-N},\ \gamma \in \Bigl(1, \frac{p+k-1}{p}\Bigr)
\]
for sufficiently large $p$ because 
$\lim_{p \nearrow \infty}\frac{2(p+1)-N}{2p-N}\cdot \frac{p}{p+1}=1$ and $\lim_{p \nearrow \infty} 
\gamma \cdot \frac{p}{p+1}=1$, we can derive from Young's inequality that 
\begin{align*}
   \Psi'(t)
   \le c_7 \Psi(t) - c_8 \Psi^\frac{p+k-1}{p}(t) 
\end{align*}
and therefore ODI comparison yields uniform bound for $\Psi(t)$ with sufficiently large $p>\frac{N}{2}$. 
Consequently, Lemma \ref{LemmaBoundedness u nabla v}  proves that $T_{max}=\infty$ and $\|u(\cdot,t)\|_{L^{\infty}(\Omega)} 
\le c_9$ for all $t>0$.  
\hfill \qedsymbol

\section*{Acknowledgments}
M. Marras and S. Vernier-Piro are members of the Gruppo Nazionale per l'Analisi Matematica, la Probabilit$\grave{\rm a}$ e le loro Applicazioni (GNAMPA) of the Istituto Nazionale di Alta Matematica (INdAM). 

\subsection*{Financial disclosure}

M. Marras is partially supported by the research project {\em Analysis of PDEs in connection with real phenomena}, CUP F73C22001130007, funded by {Fondazione di Sardegna} (2022), by the research project: {\em Evolutive and stationary Partial Differential Equations with a focus on biomathematics},  funded by Fondazione di Sardegna (2019); by the grant PRIN n. PRIN-2017AYM8XW: {\em Non-linear Differential Problems via Variational, Topological and Set-valued Methods}
and by the grant   INDAM-GNAMPA Project, CUP E55F22000270001.\\
T. Yokota is partially supported by JSPS KAKENHI Grant Number JP21K03278.

  \end{document}